\documentclass{amsart}
\usepackage{amsfonts}
\usepackage{amsmath}
\usepackage{amssymb}
\usepackage{amsthm}
\usepackage{graphicx}
\usepackage{enumerate}
\usepackage[numbers]{natbib}

\usepackage{tikz}
\usetikzlibrary{matrix,arrows,patterns}

\newcommand{\pattern}[4]{
 \raisebox{0.6ex}{
 \begin{tikzpicture}[scale=0.35, baseline=(current bounding box.center), #1]
   \foreach \x/\y in {#4}
     \fill[pattern= north east lines, pattern color=black!50] (\x,\y)
rectangle +(1,1);
   \draw (0.01,0.01) grid (#2+0.99,#2+0.99);
   \foreach \x/\y in {#3}
     \filldraw (\x,\y) circle (6pt);
 \end{tikzpicture}}\;
}

\renewcommand{\S}[1]{\ensuremath{\mathfrak{S}_{#1}}} 

\newcommand{\Sav}[2]{\S{#1}(#2)} 

\newcommand{\sav}[2]{\ensuremath{ \bigl| \Sav{#1}{#2} \bigr|}}

\newcommand{\SP}[1]{\ensuremath{\mathfrak{F}_{#1}}}

\newcommand{\SPav}[2]{\SP{#1}(#2)}

\newcommand{\SPH}[2]{\ensuremath{\mathfrak{F}_{#1}^{#2}}}

\newcommand{\SPHav}[3]{\SPH{#1}{#2}(#3)}

\newcommand{\sphav}[3]{\ensuremath{ \bigl| \SPHav{#1}{#2}{#3} \bigr|}}

\newcommand{\we}{\sim}

\newcommand{\swe}{\ensuremath{\stackrel{s}{\sim}}}  
\newcommand{\spswe}{\ensuremath{\stackrel{f\!s}{\sim}}}

\newcommand{\height}[1]{ ht( {#1} ) }
\newcommand{\width}[1]{ wd( {#1} ) }

\newcommand{\oi}{\cong}

\newcommand{\red}{\ensuremath{\mathrm{red}}}

\newtheorem{theorem}{Theorem}
\newtheorem{definition}{Definition}
\newtheorem{lemma}{Lemma}
\newtheorem{proposition}{Proposition}
\newtheorem{corollary}{Corollary}
\newtheorem{conjecture}{Conjecture}

\newtheorem{remark}{Remark}

\newtheorem*{theorem*}{Theorem}
\newtheorem*{corollary*}{Corollary}
\newtheorem*{proposition*}{Proposition}
\newtheorem*{conjecture*}{Conjecture}

\renewcommand{\d}{\makebox[1.1ex]{\rule[.58ex]{.71ex}{.15ex}}}

\newcommand{\n}{\square}

\newif\ifgraph
\graphtrue

\begin{document}

\title{Shape-Wilf-equivalences for vincular patterns}
\author{Andrew M. Baxter}
\address{
Penn State University Mathematics Dept. \\
University Park, State College, PA 16802 \\
(814) 865-5002 / Fax (814) 865-3735 \\
}

\begin{abstract}
We extend the notion of shape-Wilf-equivalence to vincular patterns (also known as ``generalized patterns'' or ``dashed patterns'').  First we introduce a stronger equivalence on patterns which we call filling-shape-Wilf-equivalence.  When vincular patterns $\alpha$ and $\beta$ are filling-shape-Wilf-equivalent, we prove that $\alpha\oplus\sigma$ and $\beta\oplus\sigma$ must also be filling-shape-Wilf-equivalent.  We also discover two new pairs of patterns which are filling-shape-Wilf-equivalent: when $\alpha$, $\beta$, and $\sigma$ are nonempty consecutive patterns which are Wilf-equivalent, $\alpha\oplus\sigma$ is filling-shape-Wilf-equivalent to $\beta\oplus\sigma$; and for any consecutive pattern $\alpha$, $1\oplus\alpha$  is filling-shape-Wilf-equivalent to $1\ominus\alpha$.  These new equivalences imply many new Wilf-equivalences for vincular patterns.
\end{abstract}


\maketitle

\section{Introduction}

Let $[n]=\{1,2,\dotsc,n\}$ and define the \emph{reduction} of a word $w= w_1 w_2 \dotsm w_k \in [n]^k$, denoted $\red(w)$, to be the word in $[k]^k$ where the $i^{th}$ smallest letter(s) of $w$ is replaced by $i$.  For example $\red(51373)= 31242$.  In the case where $w$ has no repeated letters, $\red(w)$ is a permutation in $\S{k}$.  If $\red(w) = \red(w')$ we write $w\oi w'$ and say that $w$ and $w'$ are \emph{order-isomorphic}.  Equivalently, one can say that $w\oi w'$ if for every pair of indices $(i,j)$ $w_i \leq w_j$ if and only if $w'_i \leq w'_j$.

Permutation $\pi\in\S{n}$ \emph{contains} $\sigma\in\S{k}$ \emph{as a classical pattern} if there is a subsequence $\pi_{i_1}\pi_{i_2}\dotsm\pi_{i_k}$ for $1\leq i_1 < i_2 < \dotsm < i_k\leq n$ such that  $\pi_{i_1}\pi_{i_2}\dotsm\pi_{i_k}\oi\sigma$.  Vincular patterns\footnote{These were introduced as ``generalized patterns'' in \cite{Babson2000} but have also been called ``dashed patterns.''   Claesson coined the term ``vincular'' to underline the connection with the recently introduced bivincular patterns in \cite{Bousquet2010}.} resemble classical patterns, except some of the indices $i_j$ must be consecutive.  Formally, we can consider a vincular pattern as a pair $(\sigma, X)$ for permutation $\sigma\in\S{k}$ and a set of adjacencies $X\subseteq [k-1]=\{1,\ldots, k-1\}$.   The subsequence $\pi_{i_1}\pi_{i_2}\dotsm\pi_{i_k}$ for $1\leq i_1 < i_2 < \dotsm < i_k\leq n$  is a \emph{copy} (or \emph{occurrence}) of $(\sigma, X)$ if $\pi_{i_1}\pi_{i_2}\dotsm\pi_{i_k}\oi\sigma$ \emph{and} $i_{j+1}-i_{j} = 1$ for each $j\in X$.  If a copy of $(\sigma, X)$ appears in $\pi$, we say that $\pi$ \emph{contains} $(\sigma, X)$, and otherwise we say $\pi$ \emph{avoids} $(\sigma, X)$.
In practice we write $(\sigma, X)$ as the permutation $\sigma$ with a dash between $\sigma_j$ and $\sigma_{j+1}$ if $j\not\in X$ and refer to ``the vincular pattern $\sigma$'' without explicitly writing $X$.  For example, $(1243, \{3\})$ is written $1\d2\d43$.  The permutation $162534$ has a copy of $1\d2\d43$ as witnessed by the subsequence $1253$, but the subsequence $1254$ is not a copy of $1\d2\d43$ since the 5 and 4 are not adjacent in $\pi$.  \emph{Classical} patterns are patterns of the form $(\sigma, \emptyset)$ where no adjacencies are required, while \emph{consecutive} patterns are patterns of the form $(\sigma, [k-1])$ where the copies of $\sigma$ must appear as subfactors $\pi_i \pi_{i+1} \cdots \pi_{i+k-1} \oi \sigma$.

Classical patterns exhibit several trivial symmetries which extend to vincular patterns as well.  The \emph{reverse} of a permutation $\pi=\pi_1\pi_2 \cdots \pi_n$ is given by $\pi^r=\pi_n\pi_{n-1}\cdots\pi_1$ and the \emph{complement} by $\pi^c=(n+1-\pi_1)(n+1-\pi_2)\cdots(n+1-\pi_n)$.  For vincular pattern $(\sigma, X)$ of length $k$, the reverse of $(\sigma, X)$ is the pattern $(\sigma,X)^{r} = (\sigma^r,  \{k-x : x \in X\})$.  Thus the reverse of $13\d2\d4$ is $4\d2\d31$.  Similarly the complement of $(\sigma, X)$ is the pattern $(\sigma,X)^{c} = (\sigma^c, X)$.  Thus we see $(13\d2\d4)^c = 42\d3\d1$.  Observe that $\pi$ contains $(\sigma, X)$ if and only if $\pi^r$ contains $(\sigma, X)^r$ and likewise for the complement.   

The subset of $\S{n}$ of permutations avoiding $\sigma$ is denoted $\Sav{n}{\sigma}$.   Two patterns $\sigma$ and $\tau$ are \emph{Wilf-equivalent} if $\sav{n}{\sigma} = \sav{n}{\tau}$ for all $n\geq 0$, and we denote this $\sigma \we \tau$.  From the preceding remarks on symmetry it is clear that $\sigma \we \sigma^{r} \we \sigma^{c} \we \sigma^{rc}$.

\subsection{Fillings and Shape-Wilf-Equivalence}

We will consider permutations as configurations of non-attacking rooks on a partial chessboard.  To ease later exposition we break with convention and orient our boards differently. 

A \emph{board} is a finite subset of $\mathbb{Z}^2$, where each element of the board is called a \emph{cell} and depicted as a box.  A \emph{Young board} is a board where columns have weakly decreasing heights from left to right and rows have weakly decreasing widths from bottom to top (note that we use the French orientation).   Without loss of generality, we may assume the bottom-left corner cell of a nonempty Young board is $(1,1)$ and index cells as per standard Euclidean coordinates.  We say $(c',r')$ is \emph{to the left of} $(c,r)$ if $c'<c$ and \emph{below} $(c,r)$ if $r'<r$.   A board is sometimes called a \emph{shape}, and we use the two terms interchangeably.  In the sequel, all boards discussed are Young boards.  

Let $\lambda = (\lambda_1, \lambda_2, \dotsc, \lambda_n)$ where $\lambda_i \geq \lambda_{i+1}$ and $\lambda_n>0$ denote a Young board  $\{ (i, j): 1 \leq j \leq \lambda_i, 1\leq i \leq n \}$.  Note the change from standard convention: column $i$ has height $\lambda_i$, while row $j$ has width $\lambda^{\top}_j$ where $\lambda^{\top}$ is the conjugate partition of $\lambda$.  This change in convention makes much of the later exposition easier.   Figure \ref{fig:Young55433} shows the Young board $\lambda=(5,5,4,3,3)$.
Note that $\lambda$ is a Young board if and only if $(c, r) \in \lambda$ implies $(c', r')\in\lambda$ for every $1\leq r' \leq r$ and $1\leq c' \leq c$.  The \emph{height} of the Young board $\lambda = (\lambda_1, \dotsc, \lambda_n)$ is $\lambda_1$, denoted $\height{\lambda}$, and the \emph{width} is $n$, denoted $\width{\lambda}$. A Young board $(\lambda_1, \dotsc, \lambda_n)$ is \emph{rectangular} if $\lambda_1 = \dotsm =\lambda_n$.

\ifgraph
\begin{figure} 
          \centering
	\includegraphics[width=.3\textwidth]{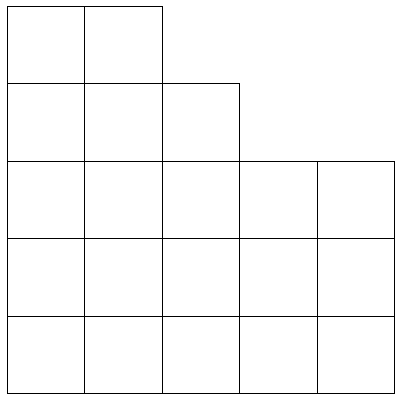}
	\caption{The Young board $\lambda=(5,5,4,3,3)$}
	\label{fig:Young55433}
\end{figure}
\fi

A \emph{filling} of a board $\lambda$ is an assignment of the integers 0 and 1 to the cells of $\lambda$ such that each row and each column contains at most one 1.\footnote{In some works, such as \cite{Rubey2011}, this would be called a \emph{partial 01-filling}.}  A \emph{standard filling} is a filling such that each row and column has exactly one 1 in it.  In diagrams, we will draw a cell filled with a 1 as a $\bullet$ in the box while leaving empty cells filled with a 0.  Thus we call a column or row \emph{empty} if it contains no 1.


Let $\mathbb{P}$ be the set of positive integers.  For Young board $(\lambda_1, \dotsc, \lambda_n)$, any filling can be represented uniquely by a word in $w \in (\mathbb{P} \cup \{\n\})^n$ according to the following rules:
\begin{itemize}
 \item If cell $(c, r)$ is filled with a 1, then $w_c = r$.
 \item If the $c^{th}$ column of $\lambda$ is empty then $w_c = \n$.
\end{itemize}
As a consequence of this representation, if the $r^{th}$ row of $\lambda$ is empty, then $r$ does not appear in the word $w$, and thus $w$ indicates both the empty rows and columns.

Conversely a word $w \in ( \mathbb{P} \cup \{\n\})^n$ represents a filling of $(\lambda_1, \dotsc, \lambda_n)$ if and only if the following criteria are satisfied:
\begin{enumerate}[(1)]
 \item If there are distinct indices $i$ and $j$ such that $w_i = w_j$, then $w_i = w_j = \n$.
 \item If $w_i \neq \n$, then $w_i \leq \lambda_i$.
\end{enumerate}
Criterion (1) implies that the numerical letters of $w$ (i.e., those letters other than $\n$) form a permutation of a subset of $[n]$.  We call the symbol $\n$ a \emph{spacer} and call words satisfying criterion (1) \emph{spaced permutations}.

Vincular pattern containment extends naturally to spaced permutations (and thus to fillings) using the same definitions but working in the poset $\mathbb{P}+\{\n\}$ instead of $\mathbb{P}$ to determine order-isomorphism, where $\n$ is incomparable to each number $x\in \mathbb{P}$.  One can accordingly extend the reduction operator to spaced permutations so that $\red(w)$ preserves the positions of the spacers, e.g., $\red(51\n37\n\n3) = 31\n24\n\n2$. Thus we see the spaced permutation $412\n5$ contains a copy of $2\d1\d3$, but avoids $2\d13$ because of the position of the spacer.  In general, the presence of a spacer in a spaced permutation affects the presence of a pattern by affecting whether letters are adjacent.\footnote{Spaced permutations first appeared implicitly in connection to enumeration schemes for vincular patterns in \cite{Baxter2012}, where a spacer was called a ``null symbol'' and denoted $\bullet$.  While spaced permutation bear some superficial resemblence to partial permutations discussed in \cite{Claesson2011}, in that work a ``hole'' (denoted by $\diamond$) could be replaced by any number when considering pattern avoidance, while in a spaced permutation $\n$ is inert and incomparable to other letters.}  

Let $\SP{\lambda}$ denote the set of fillings of board $\lambda$, and $\S{\lambda}$ denote the set of standard fillings of $\lambda$.  Observe that a filling of $\lambda$ is a standard filling if and only if its word representation contains no copies of $\n$ and each letter $\pi_i \leq \lambda_i$.  Furthermore a board $(\lambda_1, \dotsc, \lambda_n)$ admits a standard filling if and only if each $\lambda_i$ satisfies $n-i+1 \leq \lambda_i \leq n$.   Figure \ref{fig:trans} illustrates the filling $453\n21$ for Young board $(5,5,4,3,3,3)$.  In general we will not draw a distinction between a filling and its word representation.

\ifgraph
\begin{figure}
	\centering
	\includegraphics[width=.3\textwidth]{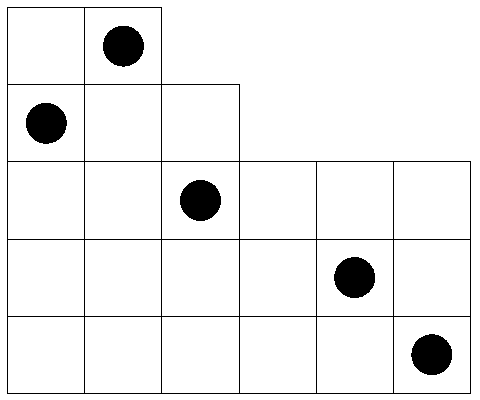}
	\caption{The filling $\pi=453\n21$ of $\lambda=(5,5,4,3,3,3)$.  Cells filled with a $1$ are marked with a $\bullet$ and cells filled with a $0$ are left blank..}
	\label{fig:trans}
\end{figure}
\fi

Pattern containment is stricter for fillings than it is for permutations and depends on the board $\lambda$.  
For board $\lambda$, a filling $\pi$  \emph{contains a copy of vincular pattern $\sigma$} if there is a subsequence $\pi_{i_1}\pi_{i_2}\dotsm\pi_{i_k}$ for $1\leq i_1 < i_2 < \dotsm < i_k\leq n$ such that  $\pi_{i_1}\pi_{i_2}\dotsm\pi_{i_k}$ is a copy of vincular pattern $\sigma$ in spaced permutation $\pi$ and $\lambda$ contains all cells in $\{ i_1, \dotsc, i_k\} \times \{\pi_{i_1}, \dotsc, \pi_{i_k}\}$.   If $\lambda$ is a Young board, this second criterion is equivalent to the simpler requirement that the upper-right corner $(i_k, \max \{\pi_{i_1},\pi_{i_2},\cdots,\pi_{i_k}\})$ lies in $\lambda$.
 For example in Figure \ref{fig:trans} we see the filling $\pi=453\n21$ in $(5,5,4,3,3,3)$ contains $3\d21$ in the last three entries since $\pi_3 > \pi_5 > \pi_6$ and $\lambda$ contains $(6,3)$.  The filling $\pi=453\n21$ does not contain $231$, even though $\red(\pi_1 \pi_2 \pi_3)=231$, since $\lambda$ does not contain the cell $(3,5)$.  We let $\SPav{\lambda}{\sigma}$ denote the set of fillings of $\lambda$ which avoid $\sigma$, and similarly let $\Sav{\lambda}{\sigma}$ denote the set of standard fillings of $\lambda$ which avoid $\sigma$.  

Note that we may consider $\Sav{\lambda}{\sigma}\subseteq \S{n}$ where $n=\width{\lambda}=\height{\lambda}$ by considering the word representation of a standard filling of $\lambda$.  Furthermore $\Sav{\lambda}{\sigma} \supseteq \Sav{n}{\sigma} \cap \S{\lambda}$ since a standard filling of $\lambda$ may avoid $\sigma$ while the word representation contains $\sigma$ as a permutation.
Next, observe that the definition for pattern containment in a filling of a rectangular board is equivalent to pattern containment in the word representation.  Therefore for any Young board $\lambda$ we see that $\SPav{\lambda}{\sigma} \supseteq \SPav{\lambda^*}{\sigma} \cap \SP{\lambda}$ where $\lambda^*$ is the rectangular board with height $\height{\lambda}$ and width $\width{\lambda}$.

In order to proceed, we must refine our sets of fillings of $\lambda$ according to the empty rows and columns.  For a set $C \subseteq [\width{\lambda}]$ of columns and $R\subseteq [\height{\lambda}]$ of rows,  let $\SPH{\lambda}{C,R}$ be the set of fillings such that column $c$ is empty if and only if $c\in C$ and row $r$ is empty if and only if $r\in R$.  Furthermore, let $\SPHav{\lambda}{C,R}{\sigma}$ be those fillings in $\SPH{\lambda}{C,R}$ which avoid the vincular pattern $\sigma$, i.e., $\SPHav{\lambda}{C,R}{\sigma} := \SPH{\lambda}{C,R} \cap \SPav{\lambda}{\sigma}$.

We are now ready for the central definitions of this paper.  The first definition is standard, while the second definition is a natural extension to fillings.
\begin{definition}
 Two patterns $\sigma$ and $\tau$ are called \emph{shape-Wilf-equivalent} if $\sav{\lambda}{\sigma}=\sav{\lambda}{\tau}$ for all Young boards $\lambda$; we denote this $\sigma\swe\tau$.

  Two patterns $\sigma$ and $\tau$ are called \emph{filling-shape-Wilf-equivalent} if $\sphav{\lambda}{C,R}{\sigma}=\sphav{\lambda}{C,R}{\tau}$ for all Young boards $\lambda$ and sets $C\subseteq [\width{\lambda}]$ and $R\subseteq [\height{\lambda}]$; we denote this $\sigma\spswe\tau$.
\end{definition}

Clearly shape-Wilf-equivalence implies Wilf-equivalence since $\Sav{n}{\sigma}$ is the same as the set of standard fillings of the $n\times n$ board which avoid $\sigma$.  Also filling-shape-Wilf-equivalence implies shape-Wilf-equivalence, since $\SPHav{\lambda}{\emptyset, \emptyset}{\sigma} = \Sav{\lambda}{\sigma}$.   Note that shape-Wilf-equivalence is such a strong condition that neither $\sigma \swe \sigma^r$ nor $\sigma \swe \sigma^{c}$ is necessarily true.

For a Young board $\lambda$, $C\subseteq \width{\lambda}$, and $R\subseteq \height{\lambda}$, let $\lambda^{(C,R)}$ be the Young board formed by deleting from $\lambda$ each column in $C$ and each row in $R$.  There is a natural bijection $\rho_{\lambda}^{C,R}: \SPH{\lambda}{C,R} \to \S{\lambda^{(C,R)}}$, which maps any filling $\pi$ of $\lambda$ in $\SPH{\lambda}{C,R}$ to a standard filling of $\lambda^{(C,R)}$ such that the numerical letters of the filling $\pi$ are order-isomorphic to the letters of the standard filling $\rho_{\lambda}^{C,R}(\pi)$.  For example, let $\lambda=(6,6,6,6,4)$, $C=\{1,3,5\}$ and $R=\{1,4,6\}$.  Then $\n3\n5\n2 \in \SPH{\lambda}{C,R}$, and $\rho_{\lambda}^{C,R}(\pi) = 231$ is a standard filling of $\lambda^{(C,R)} = (3,3,2)$.

For classical patterns, empty rows and empty columns have no effect on pattern containment since no adjacencies are required.  This can be phrased as the following proposition.
\begin{proposition}\label{prop:emptyrowsandcolumns}
 For a classical pattern $\sigma$, Young board $\lambda$, and subsets $R\subseteq [\height{\lambda}]$, and $C\subseteq [\width{\lambda}]$, $\rho_{\lambda}^{(C,R)}$ provides a bijection from $\SPHav{\lambda}{C,R}{\sigma}$ to $\Sav{\lambda^{(C,R)}}{\sigma}$.
\end{proposition}

This has the corollary that shape-Wilf-equivalence and filling-shape-Wilf-equivalence are equivalent notions for classical patterns:
\begin{corollary}\label{cor:emptyrowsandcolumns}
For classical patterns $\sigma$ and $\tau$, $\sigma \swe \tau$ if and only if $\sigma \spswe \tau$.
\end{corollary}

For vincular patterns we may only go halfway.  Empty rows make no difference to pattern containment for vincular patterns, because the values of the letters forming the copy of the pattern do not matter except insofar as their relation to one another, but empty columns affect adjacency.  Therefore we have the following proposition.
\begin{proposition}\label{prop:emptyrows}
 For a vincular pattern $\sigma$, Young board $\lambda$, and subsets $R\subseteq [\height{\lambda}]$, and $C\subseteq [\width{\lambda}]$, the map deleting the rows in $R$ provides a bijection from $\SPHav{\lambda}{C,R}{\sigma}$ to $\SPHav{\lambda^{(\emptyset, R)}}{C,\emptyset}{\sigma}$.
\end{proposition}

As an immediate corollary to Proposition \ref{prop:emptyrows} we get:

\begin{corollary}\label{cor:emptyrows}
  If  $\sigma$ and $\tau$ are vincular patterns such that $\sphav{\lambda}{C,\emptyset}{\sigma}=\sphav{\lambda}{C,\emptyset}{\tau}$ for all boards $\lambda$ and sets $C\subseteq [\width{\lambda}]$, then $\sigma \spswe \tau$.
\end{corollary}

\subsection{Direct sum and skew sum of vincular patterns}
Recall the \emph{direct sum} of two permutations, $\alpha\in\S{k}$ and $\beta\in\S{\ell}$, is the length-$(k+\ell)$ permutation 
\begin{equation}\label{eqn:defineoplus}
 \alpha \oplus \beta = \alpha_1\alpha_2\cdots\alpha_k(\beta_1+k)(\beta_2+k)\cdots(\beta_{\ell}+k).
\end{equation}
  This is most easily seen as placing a copy of $\beta$ above and to the right of $\alpha$ as fillings of an $(k+\ell)\times(k+\ell)$ rectangular board. See Figure \ref{fig:sum} where $312\oplus2413=312\,5746$ is illustrated.

\ifgraph
\begin{figure}
	\centering
	\includegraphics[width=.3\textwidth]{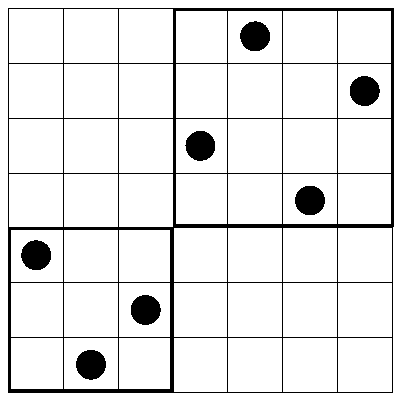} \qquad \includegraphics[width=.3\textwidth]{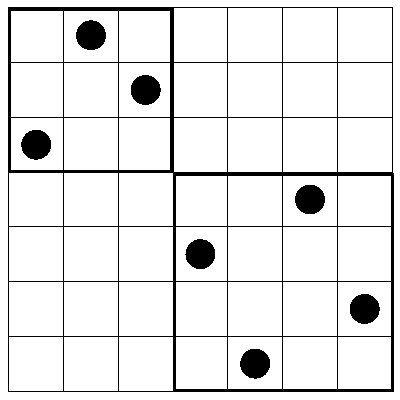}
	\caption{Illustrations of $312\oplus2413=3125746$ (left) and  $132\ominus3142=5763142$ (right).}
	\label{fig:sum}
\end{figure}
\fi

\begin{definition}
The \emph{direct sum} of two vincular patterns $(\sigma, A)\in\S{k}\times [k-1]$ and $(\tau, B)\in\S{\ell}\times [\ell-1]$ is the pattern $(\sigma,A)\oplus (\tau,B) = (\sigma\oplus \tau, C)$ where $C=A \cup \{i+k: i\in B\}$. 
\end{definition}

 For example, $(31\d2)\oplus(24\d13)=31\d2\d57\d46$.  Observe that this definition is equivalent to the classical definition when $\sigma$ and $\tau$ are classical patterns.

Similarly, the \emph{skew sum} of two permutations, $\alpha\in\S{k}$ and $\beta\in\S{\ell}$, is the length-$(k+\ell)$ permutation 
\begin{equation}\label{eqn:defineominus}
 \alpha \ominus \beta = (\alpha_1+\ell)(\alpha_2+\ell)\dotsm (\alpha_k+\ell)\beta_1\beta_2\dotsm\beta_{\ell}.
\end{equation}
 This is most easily seen as placing $\beta$ below and to the right of $\alpha$ as fillings of a rectangular board. For example, $132\ominus3142=576\,3142$ is illustrated in Figure \ref{fig:sum}.

\begin{definition}
The \emph{skew sum} of two vincular patterns $(\sigma, A)\in\S{k}\times [k-1]$ and $(\tau, B)\in\S{\ell}\times [\ell-1]$ is the pattern $(\sigma,A)\ominus (\tau,B) = (\sigma\ominus \tau, C)$ where $C=A \cup \{i+k: i\in B\}$ (as in the case for direct sum). 
\end{definition}

 For example, $(13\d2)\ominus(31\d42)=57\d6\d31\d42$.  As for direct sum, this definition is equivalent to the classical definition when $\sigma$ and $\tau$ are classical patterns.

\subsection{Summary of results}
In Section \ref{sec:SWEandWE}, we will prove the following result:

\begin{theorem*}
For vincular patterns $\alpha$, $\beta$, and $\sigma$, if $\alpha\spswe\beta$ then $\alpha\oplus\sigma \spswe \beta\oplus\sigma$.  
\end{theorem*}

This theorem generalizes the celebrated result by Backelin, West, and Xin \cite{Backelin2007}, which says that if $\alpha$, $\beta$, and $\sigma$ are \emph{classical} patterns, then $\alpha \swe \beta$ implies $\alpha\oplus\sigma \swe \beta\oplus\sigma$.

In Section 3, we will prove two new filling-shape-Wilf-equivalences.  In Subsection \ref{sec:oplussigma}  we will prove:
\begin{theorem*}
  Let $\alpha$, $\beta$, and $\sigma$ be nonempty consecutive patterns.  Then $\alpha\we\beta$ implies that $\alpha\oplus\sigma \spswe \beta\oplus\sigma$.
\end{theorem*}

In Subsection \ref{sec:1-23swe3-12} we will prove:
\begin{theorem*}
 For any consecutive pattern $\sigma$, $1\oplus\sigma \spswe 1\ominus\sigma$.
\end{theorem*}

Subsections \ref{sec:oplussigma} and \ref{sec:1-23swe3-12} also list the new Wilf-equivalences for vincular patterns of length 4 and 5.  Conjectures are collected in Section \ref{sec:Conclusion}.

\section{Shape-Wilf-equivalence and Wilf-equivalence}\label{sec:SWEandWE}

In \cite{Backelin2007} Backelin, West, and Xin prove the following proposition for classical patterns $\alpha$, $\beta$, and $\sigma$.  

\begin{proposition*}[Backelin, West, Xin \cite{Backelin2007}]
For classical patterns $\alpha$, $\beta$, and $\sigma$, if $\alpha\swe\beta$ then $\alpha\oplus\sigma \swe \beta\oplus\sigma$.  
\end{proposition*}

In this section we adapt their proof to obtain the following theorem:

\begin{theorem} \label{thm:spswe}
For vincular patterns $\alpha$, $\beta$, and $\sigma$, if $\alpha\spswe\beta$ then $\alpha\oplus\sigma \spswe \beta\oplus\sigma$.  
\end{theorem}

\begin{proof}
 Let $\alpha \spswe \beta$, and for each board $\lambda$ and subset $C\subseteq [\width{\lambda}]$ and $R\subseteq [\height{\lambda}]$ let $f_{\lambda}^{C,R}:\SPHav{\lambda}{C,R}{\alpha}\to\SPHav{\lambda}{C, R}{\beta}$ be a bijection (as guaranteed by the hypothesis). We will use these bijections to construct a bijection $g_{\lambda}^{C,R}: \SPHav{\lambda}{C,R}{\alpha\oplus\sigma} \to \SPHav{\lambda}{C, R}{\beta\oplus\sigma}$ by coloring the cells of $\lambda$ either white or gray, then applying the transformation within the white cells while leaving the contents of the gray cells fixed.

Fix board $\lambda$, and sets $C\subseteq[\width{\lambda}]$ and $R\subseteq[\height{\lambda}]$.  Let $\pi\in\SPHav{\lambda}{C,R}{\alpha\oplus\sigma}$.  Color the cell $(c, r)$ of $\lambda$ white if the filling of the subboard (strictly) above and to the right of it contains $\sigma$, or gray otherwise.  See Figure \ref{fig:combined}.

\ifgraph
\begin{figure}[htb]
	\centering
		\includegraphics[width=.4\textwidth]{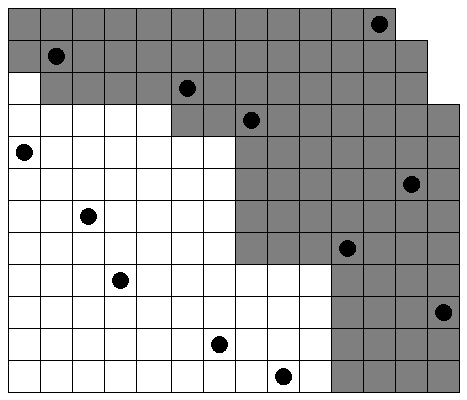}
	\caption{Coloring for filling $\pi = 8(11)64\n(10)291\n5(12)73$ for board $\lambda = (12^{12},11,9)$, where $\sigma=1$-$2$.  The white cells form subboard $\bar{\lambda} = (10, 9,9,9,9, 8,8,4,4,4)$ with filling $\bar{\pi} =8\n64\n\n2\n1\n$ .}
	\label{fig:combined}
\end{figure}
\fi

  The white cells form a Young board, denoted $\bar{\lambda}$, since if a cell is colored white then so is every cell to its left and below it.  The filling $\pi$ restricts to a filling $\bar{\pi}$ of $\bar{\lambda}$ with a set of empty columns $\bar{C}\supseteq C\cap[\width{\bar{\lambda}}]$ and a set of empty rows $\bar{R}\supseteq R\cap[\height{\bar{\lambda}}]$.  
   The filling for $\pi$ avoids $\alpha\oplus\sigma$, and all white cells of $\lambda$ lie below and to the left of a copy of $\sigma$, therefore $\bar{\pi} \in \SPHav{\bar{\lambda}}{\bar{C},\bar{R}}{\alpha}$.  Apply $f_{\bar{\lambda}}^{\bar{C}, \bar{R}}$ to $\bar{\pi}$, which will leave a filling of $\bar{\lambda}$ which avoids $\beta$ and has the same empty rows and columns as $\bar{\pi}$.  Restoring the gray cells of $\lambda$ and their contents, we are left with a filling of $\lambda$ which avoids $\beta\oplus\sigma$.  

To see that this procedure is invertible, it suffices to note that the application of the forward map only changes the contents of the white cells.  Thus when applying the inverse map, the coloring of $\lambda$ according to copies of $\sigma$ will be the same, resulting in the same subboard $\bar{\lambda}$ on which to apply the inverse transformation $\bigl(f_{\bar{\lambda}}^{\bar{C}, \bar{R}}\bigr)^{-1}$.
\end{proof}


\section{New Filling-Shape-Wilf-Equivalent Pairs}

Theorem \ref{thm:spswe} is vacuous without pairs $\alpha \spswe \beta$.  As of this writing, there is an infinite family of shape-Wilf-equivalences in \cite{Backelin2007}, and another known pair of shape-Wilf-equivalent classical patterns of length 3 in \cite{Stankova2002}.  We will begin by listing these classical results, which extend to filling-shape-Wilf-equivalence by Corollary \ref{cor:emptyrowsandcolumns}.

\begin{theorem}[Backelin, West, and Xin~\cite{Backelin2007}]\label{thm:BWX}
 Let $J_t = t\d(t-1)\d\cdots\d2\d1$.  Then  $J_t \spswe J_{t-1}\oplus 1$.
\end{theorem}

\begin{corollary}[Backelin, West, and Xin~\cite{Backelin2007}]\label{cor:BWX}
 Let $I_t = 1\d2\d\cdots\d(t-1)\d t$.  Then $J_t \spswe J_k \oplus I_{t-k}$ for any $0\leq k \leq t$.  In particular, $J_t \spswe I_t$.
\end{corollary}


\begin{theorem}[Stankova and West \cite{Stankova2002}]\label{thm:SW}
 $2\d3\d1 \spswe 3\d1\d2$.
\end{theorem}

We now consider non-classical vincular patterns.    It can be seen by computing each $\sav{\lambda}{\sigma}$ for $\lambda=(5,5,5,5,4)$ that there is no shape-Wilf-equivalent pair (and thus no filling-shape-Wilf-equivalent pair) of consecutive patterns of length 3, as shown by Table \ref{tab:consecutive3}.

\ifgraph
\begin{table}
	\centering
		\begin{tabular}{c|c}
		            $\sigma$ & $\sav{\lambda}{\sigma}$ \\
		            \hline
                $213$ & $56$ \\
                $132$ & $57$ \\
                $231$ & $58$ \\
                $312$ & $59$ \\
                $123$ & $61$ \\
                $321$ & $65$
		\end{tabular}
	\caption{$\sav{\lambda}{\sigma}$ for consecutive patterns $\sigma\in\S{3}$ and $\lambda=(5,5,5,5,4)$.}
	\label{tab:consecutive3}
\end{table}
\fi

A computer search over standard fillings of Young boards of width at most $6$ reveals there are at most three pairs of shape-Wilf-equivalent patterns of length 3 with one dash:
\begin{equation}\label{eqn:equivs}
   12\d3 \swe 21\d 3, \qquad
   3\d12 \swe 1\d 23, \qquad
   3\d21 \swe 1\d 32.
\end{equation}
Indeed each of these holds, and furthermore each of these is a filling-shape-Wilf-equivalence with a further-reaching generalization.

\subsection{The equivalence $12\!\!-\!\!3 \spswe 21\!\!-\!\!3$}\label{sec:oplussigma}

In this subsection we will prove the following equivalence, which strengthens an equivalence from \cite{Elizalde2006, Kitaev2005}.

\begin{theorem}\label{thm:oplussigma}
 Let $\alpha$, $\beta$, and $\sigma$ be nonempty consecutive patterns.  Then $\alpha \we \beta$ implies $\alpha\oplus \sigma \spswe \beta \oplus \sigma$.
\end{theorem}

As a special case $\alpha=12$, $\beta=21$, $\sigma=1$ we get the first equivalence from \eqref{eqn:equivs}:  $12\d3 \spswe 21\d3$.

\begin{remark}
 In \cite{Elizalde2006, Kitaev2005} it is shown that if $\alpha\we\beta$ for consecutive patterns $\alpha$ and $\beta$, then $\alpha\oplus 1 \we \beta\oplus 1$.  In \cite{Elizalde2006} it is also shown that $\alpha\we\beta$ implies $\alpha\oplus12 \we \beta\oplus 12$ and $\alpha\oplus21 \we \beta\oplus 21$, as well as $12\oplus\sigma \we 21\oplus\sigma$.  These are all special cases of Theorem \ref{thm:oplussigma} above.
\end{remark}

Our approach takes several cues from the proof of Proposition 5.2 from \cite{Elizalde2006}.  First one must generalize the notion of a right-to-left maximum for a permutation.

\begin{definition}
 Let $\pi$ be a filling of Young board $\lambda$, and let $\sigma = \sigma_1 \dotsm \sigma_k$ be a consecutive pattern.  The subfactor $\pi_i \pi_{i+1} \dotsm \pi_{i+k-1}$ is a \emph{right-to-left maximal copy of $\sigma$} if the following criteria are satisfied:
 \begin{enumerate}
  \item  $\pi_i \pi_{i+1} \dotsm \pi_{i+k-1} \oi \sigma$.
  \item If $j>i$ and $\pi_j \pi_{j+1} \dotsm \pi_{j+k-1} \oi \sigma$ and $\sigma_m = \min(\sigma_1, \dotsc, \sigma_k)$, then $\pi_{j+m-1} < \pi_{i+m-1} $.  In other words, the minimal letter of $\pi_i \pi_{i+1} \dotsm \pi_{i+k-1}$ is greater than the minimal letter of any other copy of $\sigma$ to starting the right of $\pi_i$.
 \end{enumerate}
\end{definition}

For example, consider the standard filling illustrated in Figure \ref{fig:Fig5}.  The right-to-left maximal copies of $12$ start at columns $6$, $11$, $15$, and $19$.

Observe that if  $\pi_i \pi_{i+1} \dotsm \pi_{i+k-1}$ is a copy of the consecutive pattern $\sigma$ which is not right-to-left maximal, then $\pi$ has another copy of $\sigma$ with its first letter to the right of $\pi_i$ and its minimal letter greater than $\min( \pi_i, \pi_{i+1}, \dotsc, \pi_{i+k-1})$.  Therefore if $\pi$ contains a copy of $\alpha\oplus\sigma$, then $\pi$ contains a copy of $\alpha\oplus\sigma$ such that the part corresponding to $\sigma$ is a right-to-left maximal copy of $\sigma$.   Also observe that if $\sigma=1$, the right-to-left maximal copies of $\sigma$ in $\pi$ are exactly the right-to-left maxima of $\pi$.

To proceed we will make use of the following lemma regarding fillings of rectangular boards avoiding a consecutive pattern.
\begin{lemma}\label{lem:consrow}
  If $\alpha \we \beta$ for consecutive patterns $\alpha$ and $\beta$, then $\sphav{\lambda}{C, R}{\alpha} = \sphav{\lambda}{C, R}{\beta}$ for any rectangular board $\lambda$, $C\subseteq [\width{\lambda}]$, and $R\subseteq [\height{\lambda}]$.
\end{lemma}

\begin{proof}
By Proposition \ref{prop:emptyrows}, it suffices to prove  $\sphav{\lambda}{C, \emptyset}{\alpha} = \sphav{\lambda}{C, \emptyset}{\beta}$ for rectangular boards $\lambda$.  Recall that pattern avoidance for any filling of a rectangular board is equivalent to pattern avoidance for its word representation.  Let $\pi$ be a spaced permutation with spacers in the positions given by $C$ and decompose $\pi$ into $w_1 \n w_2 \n \dotsm \n w_m$ for (possibly empty) words $w_i$ which do not contain spacers.  Each of these subfactors $w_i$ can be considered a filling of a rectangular board with width $|w_i|$ and height $\height{\lambda}$ with no empty columns and $\height{\lambda} - |w_i|$ empty rows.

 By the definition of pattern containment for consecutive patterns, it is clear that the spaced permutation $\pi$ avoids $\alpha$ if and only if each of the subfactors $w_i$ avoids $\alpha$.   Since $\alpha\we\beta$, there are bijections $f_{n}: \Sav{n}{\alpha}\to\Sav{n}{\beta}$ for each $n\geq 0$.  Proposition \ref{prop:emptyrows} shows these bijections extend to $f^{R}_n: \SPH{\mu_n}{\emptyset, R} \to \SPH{\mu_n}{\emptyset, R}$ where $\mu_n$ is the rectangle with height $\height{\lambda}$ and width $n$.  Apply the appropriate $f^{R}_n$ to each of the subfactors $w_i$ to get a $\beta$-avoiding word $w'_i$ with the same letters.  Thus $w_1 \n w_2 \n \dotsm \n w_m \mapsto w'_1 \n w'_2 \n \dotsm \n w'_m$ provides a bijection from  $\SPHav{\lambda}{C, \emptyset}{\alpha}$ to $\SPHav{\lambda}{C, \emptyset}{\beta}$.
\end{proof}

We now proceed to the proof of Theorem \ref{thm:oplussigma}.   We employ a strategy similar to the proof of Theorem \ref{thm:spswe}, where cells of $\lambda$ are colored and a transformation is applied within the white cells while the other cells remain fixed.

\begin{proof}
  We will construct a bijection from $\SPHav{\lambda}{C,R}{\alpha\oplus\sigma}$ to $\SPHav{\lambda}{C,R}{\beta\oplus\sigma}$, based on the bijection from Lemma \ref{lem:consrow}.  Fix a Young board $\lambda$ and a subset of columns $C\subseteq [\width{\lambda}]$, and let $\pi$ be a filling of $\lambda$  in $\SPHav{\lambda}{C,R}{\alpha\oplus\sigma}$.

Color the cells of $\lambda$ as follows (see Figure \ref{fig:Fig5} for an example):
\begin{enumerate}
 \item If $\pi_{i} \pi_{i+1} \dotsm \pi_{i+|\sigma|-1}$ is a right-to-left maximal copy of $\sigma$, then color the cell $(i, \min( \pi_{i}, \pi_{i+1}, \dotsc \pi_{i+|\sigma|-1} ) )$ red.

 \item Color white all cells below and to the left of a red cell.

 \item Color all remaining cells gray.
\end{enumerate}

\begin{figure}
 \centering
 \includegraphics[width=.3\textwidth]{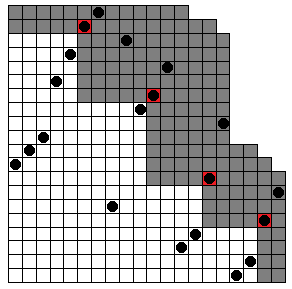}
 \caption{Coloring cells according to the proof of Theorem \ref{thm:oplussigma} for $\sigma=12$.  Red cells appear at (6, 19), (11, 14), (15, 8), (19, 5).}
  \label{fig:Fig5}
\end{figure}

By the definition of the right-to-left maximal copies of $\sigma$, no red cell lies above and to the right of any other red cell.   Thus no red cell is re-colored in Step 2.  Furthermore, the white cells form a Young board, since coloring $(c,r)$ implies $(c', r')$ is also white for any $1 \leq c' \leq c$ and $1\leq r'\leq r$.  Call this white board $\bar{\lambda}$, and observe that $\bar{\lambda}_{i-1} > \bar{\lambda}_{i}$ if and only if $(i, \lambda_{i-1}+1)$ is a red cell of $\lambda$ other than the rightmost red cell.  

We will use the red cells to dissect $\bar{\lambda}$ into rectangular subboards as follows.  Suppose that the filling $\pi$ implies red cells in columns $r_1, \dotsc, r_a$.  Then $\{r_1, \dotsc, r_a\}\subseteq \bar{C}$, since the red cells mark the column of the first letter of a copy of $\sigma$.  Let $\mu_i$ be the subboard formed by the columns $r_{i} + 1, r_i+2, \dotsc, r_{i+1}-1$ of $\bar{\lambda}$ for $2\leq i \leq a$, and let $\mu_1$ be the subboard formed by the columns $1, 2, \dotsc, r_1-1$.  Since the red cells mark columns $i$ such that $\bar{\lambda}_{i-1} \neq \bar{\lambda}_i$, we see that each of the $\mu_i$ is rectangular.

Let $\bar{\pi}^{(i)}$ be the restriction of $\pi$ to $\mu_i$, resulting in empty rows $\bar{R}^{(i)}$ and empty columns $\bar{C}^{(i)}$.  Any copy of $\alpha$ in the filling $\bar{\pi}^{(i)}$ corresponds to a copy of $\alpha$ in $\pi$ lying below and to the right of a red cell.  That cell was colored red because it is the bottom-left cell of a (right-to-left maximal) copy of $\sigma$ in $\pi$, which implies a copy of $\alpha\oplus\sigma$ in the original filling $\pi$.  Therefore $\bar{\pi}^{(i)}$ (as a filling of $\mu_i$) must avoid $\alpha$.  Lemma \ref{lem:consrow} provides a bijection for each $\bar{\pi}^{(i)}$ to a $\beta$-avoiding filling of $\mu_i$ with the same empty rows and columns.  Performing these bijections on each $\bar{\pi}^{(i)}$ and preserving the contents of the red and gray cells, we arrive at a bijection $\SPHav{\lambda}{C,R}{\alpha\oplus\sigma} \to \SPHav{\lambda}{C,R}{\beta\oplus\sigma}$.
\end{proof}

\subsubsection{Implications for Wilf-classification}

We will first restate the equivalence from \eqref{eqn:equivs} which is an immediately corollary of Theorem \ref{thm:oplussigma} and the Wilf-equivalence $12 \we 21$.

\begin{corollary}\label{cor:12-3swe21-3}
 We have the filling-shape-Wilf-equivalence $12\d3 \spswe 21\d3$.
\end{corollary}

Theorem \ref{thm:spswe}, combined with Theorem \ref{thm:oplussigma}, Theorem \ref{thm:BWX}, and the trivial symmetries, has several consequences for the Wilf-classification for vincular patterns of length 4 and 5 which we summarize in Corollary \ref{cor:Wilf45}.

\begin{corollary}\label{cor:Wilf45}
 We have the following Wilf-equivalences:
 \begin{enumerate}[(a)]
  \item $123\d4 \we 321 \d 4$,
  \item $213\d4 \we 231\d 4 \we 132\d4 \we 312\d4$,
  \item $12\d3\d4 \we 12\d4\d3 \we 21\d3\d4 \we 21\d4\d3$,
  \item $12\d34 \we 12\d43 \we 21\d34 \we 21\d43$,
  \item $12\d345 \we 21\d345 \we 12\d543 \we 21\d543$,
  \item $12\d435 \we 12\d453 \we 12\d534 \we 12\d354 \we 21\d435 \we 21\d453 \we 21\d534 \we 21\d354$,
  \item $123\d4\d5 \we 321 \d 4\d5$,
  \item $213\d4\d5 \we 231\d 4\d5 \we 132\d4\d5 \we 312\d4\d5$,
  \item $12\d3\d4\d5 \we12\d5\d4\d3 \we 12\d3\d5\d4 \we 21\d3\d4\d5 \we 21\d5\d4\d3 \we 21\d3\d5\d4$,
  \item $12\d4\d3\d5 \we 21\d4\d3\d5$, 
  \item $12\d5\d3\d4 \we 12\d4\d5\d3 \we 21\d5\d3\d4 \we 21\d4\d5\d3$.
 \end{enumerate}
\end{corollary}

Computing $\sav{n}{\sigma}$ for $n\leq 8$ suggests that two of the above classes, (i) and (j), should be merged:

\begin{conjecture}\label{con:Conj1}
 $12\d3\d4\d5 \we 12\d4\d3\d5$
\end{conjecture}

Conjecture \ref{con:Conj1} appears as it could follow from either $12\d3\d4 \spswe 12\d4\d3$ or $1\d2\d3 \spswe 1\d3\d2$ and Theorem \ref{thm:spswe}.  Computation shows, however, that $\sav{\lambda}{12\d3\d4} \neq \sav{\lambda}{12\d4\d3}$ for $\lambda=(7,7,7,7,7,7,6)$, so $12\d3\d4 \not\spswe 12\d4\d3$.  Similarly, $\sav{\lambda}{1\d2\d3} \neq \sav{\lambda}{1\d3\d2}$ for $\lambda=(5,5,5,5,4)$, so $1\d2\d3 \not\spswe 1\d3\d2$.

\subsection{$1\!\!-\!\!23 \spswe 3\!\!-\!\!12$ and $1\!\!-\!\!32 \spswe 3\!\!-\!\!21$}\label{sec:1-23swe3-12}

In this section we prove the filling-shape-Wilf-equivalences $1\d23 \spswe 3\d12$ and $1\d32 \spswe 3\d21$ from  \eqref{eqn:equivs}  as corollaries of the following theorem:

\begin{theorem}\label{thm:plustominus}
  For any consecutive pattern $\sigma$, $1\oplus\sigma \spswe 1\ominus\sigma$.
\end{theorem}

Theorem \ref{thm:plustominus} strengthens a result in \cite{Elizalde2006, Kitaev2005}, where it is shown $1\oplus\sigma \we 1\ominus\sigma$.   Alternately, Theorem \ref{thm:plustominus} could be viewed as a generalization of $1\d2 \swe 2\d1$, which is the special case $\sigma = 1$.   Computation suggests that the 1 in Theorem \ref{thm:plustominus} cannot be extended to a longer consecutive pattern since  $\sav{\lambda}{12\oplus1} \neq \sav{\lambda}{12\ominus1}$ and $\sav{\lambda}{21\oplus12} \neq \sav{\lambda}{21\ominus12}$ for the board $\lambda = (6,6,6,6,6,5)$.

We now set about proving Theorem \ref{thm:plustominus} by constructing a bijection on the set of fillings of a Young board $\lambda$ such that any filling avoiding a pattern of the form $1\oplus\sigma$ maps to a filling avoiding $1\ominus\sigma$.

We first define a map on standard fillings of Young boards $\zeta_{\lambda}: \S{\lambda} \to \S{\lambda}$, and in Lemma \ref{lem:zetabijection} we will prove $\zeta_\lambda$ is bijective.  We construct $\zeta_{\lambda}$ recursively by $\width{\lambda}$.   
   If $\lambda$ is a Young board of width 0 or 1, let $\zeta_\lambda$ be the identity transformation.  Now assume we have constructed $\zeta_\lambda$ for all Young boards $\lambda$ for $\width{\lambda} < n$ and let $\pi \in \S{\lambda}$ be a standard filling of Young board $\lambda$ with width $n$.  We will construct the standard filling $\pi^* := \zeta_{\lambda}(\pi)$ as follows.

Let $j$ be the index so that $\pi_j=1$ (i.e., the column of the cell filled with a ``1'' in the bottom row), and decompose $\pi = h\,1\,t$ for [possibly empty] words $h$ and $t$.  Let $\pi^*_{j} = \lambda_{j}$, thus filling the top cell in column $j$ with a 1.   Next, since each letter in $t$ is at least $2$, we may set $\pi^*_i = \pi_i -1$ for $i>j$.  In terms of the filling, we move each 1 from the standard filling $\pi$ down by one row.  This fills columns $j, j+1, \dotsc, n$ for $\pi^*$.  We now consider how to fill columns $1, \dotsc, j-1$.

The word $h= \pi_1 \dotsm \pi_{j-1}$ is exactly the restriction of $\pi$ to the Young board $\mu = (\lambda_1, \dotsc, \lambda_{j-1})$.  Note $h$ is a filling of $\mu$ with no empty columns and empty rows $R= \{\pi_j, \pi_{j+1}, \dotsc, \pi_n\}$.    Recall the bijection  $\rho_{\lambda}^{C,R}: \SPH{\lambda}{C,R} \to \S{\lambda^{(C,R)}}$, which deletes the empty rows and columns of a filling.  Applying the map $\rho_{\mu}^{(\emptyset, R)}$ removes the empty rows of $h$ and produces a filling $h'$ of the Young board $\mu' = \mu^{(\emptyset, R)}$.  By induction we may define $h^{*} = \zeta_{\mu'}(h')$ since $\width{\mu'} = j-1 < \width{\lambda}$.

We now fill the first $j-1$ columns of $\lambda$ in such a way that $\pi^*_1 \dotsc \pi^*_{j-1} \oi h^*$ and no row contains more than one $1$.  To see that such filling can always be completed, let $R' = \{\lambda_j, \pi_{j+1}-1, \pi_{j+1}-1, \dotsc, \pi_n -1\}$.   Since $\lambda$ is a Young diagram, every $r\in R \cup R'$ satisfies $r\leq \lambda_{j-1}$.  It follows that $\mu^{(\emptyset, R)} = \mu^{(\emptyset, R')} = \mu'$, and so $\bigl(\rho_{\mu}^{\emptyset, R'}\bigr)^{-1}(h^{*})$ provides a filling of $\mu$ with empty rows in $R'$.  Thus we complete the filling $\pi^*$ by filling columns $1, \dotsc, j-1$ of $\lambda$ according to the filling $\bigl(\rho_{\mu}^{\emptyset, R'}\bigr)^{-1}(h^{*})$ of $\mu$.

For example, consider the Young board $\lambda=(7,7,7,7,7,6,5)$ and the standard filling $\pi = 5736214$, illustrated in Figure \ref{fig:zeta}.  Since $\pi_6 = 1$, we get that $h=57362$ and $t=4$.  We immediately get $\pi^*_6 = \lambda_6 = 6$, and $\pi^*_7 = \pi_7 -1 = 3$.  Now $\mu = (\lambda_1, \dotsc, \lambda_5) = (7,7,7,7,7)$ and $\mu' = \mu^{(\emptyset, \{1,4\})} = (5,5,5,5,5)$.  Now $h' = 35241$ is a standard filling of $\mu'$, and $h^* = \zeta_{\mu'}(35241) = 31425$.  We now apply $\bigl(\rho_{\mu}^{\emptyset, R'}\bigr)^{-1}(h^{*})$ where $R'=\{3,6\}$, and thus fill $\mu$ by $41527$.  Therefore we see the resulting $\pi^* = \zeta_\lambda(\pi) = 4152763$.

\begin{figure}
	\centering
	\includegraphics[width=\textwidth]{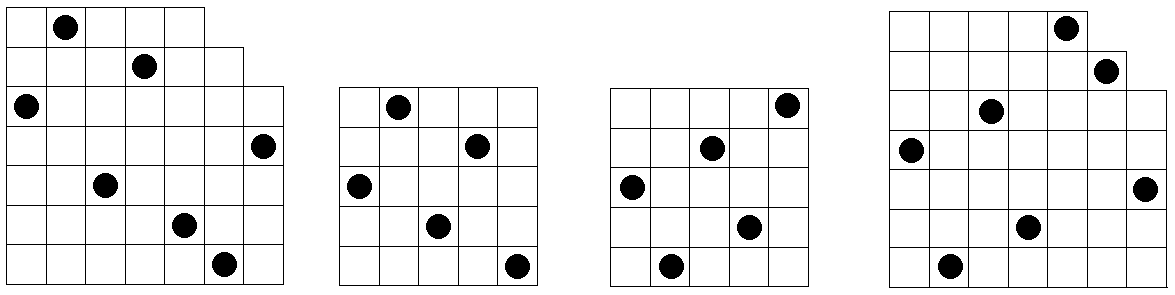}
	\caption{The action of $\zeta_\lambda$ for the filling $\pi=5736214$ of $\lambda=(7,7,7,7,7,6,5)$.  From left to right: the filling $\pi$ of $\lambda$, the filling $h'$ of $\mu'$, the filling $h^*$ of $\mu'$, the filling $\pi^*$ of $\lambda$.}
	\label{fig:zeta}
\end{figure}

The map $\zeta_\lambda$ provides a bijection on the set of standard fillings of a Young board $\lambda$ to itself, as shown in the following lemma.

\begin{lemma}\label{lem:zetabijection}
  The map $\zeta_{\lambda}: \S{\lambda} \to \S{\lambda}$ is a bijection.
\end{lemma}

\begin{proof}
It suffices to prove injectivity, since any injection from a finite set to itself is a bijection.  Suppose $\lambda$ is a board of minimum width such that there are distinct standard fillings $\pi, \pi' \in \S{\lambda}$ such that $\zeta_\lambda (\pi) = \zeta_\lambda(\pi') = \pi^*$.  Define indices $j, j'$ such that $\pi_j = 1$ and $\pi'_{j'}=1$.  If $j<j'$, then by the construction of $\zeta_{\lambda}(\pi)$ we see $\pi^{*}_{j'} = \pi_{j'}-1 < \lambda_{j'}$ while the construction of $\zeta_\lambda(\pi')$ implies $\pi^{*}_{j'} = \lambda_{j'}$.  Thus $j=j'$.  For any $i>j$, the definition of $\zeta_{\lambda}$ tells us $\pi^{*}_i = \pi_i-1$ and $\pi^{*}_i= \pi'_i-1$, which confirms that $\pi_i = \pi'_i$ for $i\geq j$.   Thus we must conclude that $h = \pi_1 \dotsm \pi_{j-1}$ and $h' = \pi'_1 \dotsm \pi'_{j-1}$ are distinct.  The definition of $\zeta_{\lambda}$, however, tells us that both $h$ and $h'$ are fillings of $\mu = (\lambda_1, \dotsc, \lambda_{j-1})$ that get mapped to $\pi^{*}_1 \dotsm \pi^{*}_{j-1}$ under the map $\bigl( \rho_{\mu}^{\emptyset, R'} \bigr)^{-1} \circ \zeta_{\mu'} \circ \rho_{\mu}^{\emptyset, R}$.  This implies that $\zeta_{\mu'}$ is not bijective, contradicting the minimality of $\lambda$.  
\end{proof}

To prove Theorem \ref{thm:plustominus}, however, we must consider the sets of fillings $\SPH{\lambda}{C,R}$, and so we extend $\zeta_\lambda$ to fillings with given empty rows and columns.  We again use the bijection  $\rho_{\lambda}^{C,R}: \SPH{\lambda}{C,R} \to \S{\lambda^{(C,R)}}$ and define $\zeta_{\lambda}^{C,R} := \bigl(\rho_{\lambda}^{C,R}\bigr)^{-1} \circ \zeta_{\lambda^{(C,R)}} \circ \rho_{\lambda}^{C,R}$ to produce a bijection $\zeta_{\lambda}^{C,R} : \SPH{\lambda}{C,R} \to \SPH{\lambda}{C,R}$.  The map $\zeta_{\lambda}^{C,R}$ can be visualized as a single map by coloring gray the cells in the empty columns and rows, coloring the remaining cells white, and then applying $\zeta_{\lambda^{(C,R)}}$ to the standard filling of the white subboard.

We state the bijectivity of $\zeta_{\lambda}^{C,R}$ as a lemma on its own.

\begin{lemma}\label{lem:zetaCRbijection}
  The map $\zeta_{\lambda}^{C,R} : \SPH{\lambda}{C,R} \to \SPH{\lambda}{C,R}$ is a bijection.
\end{lemma}

It remains to consider the effect that $\zeta_{\lambda}^{C,R}$ has on patterns of the form $1\oplus\sigma$.

\begin{lemma}\label{lem:zetaCRavoid}
 Let $\sigma$ be a consecutive pattern.  Then $\zeta_{\lambda}^{C,R} \bigl( \SPHav{\lambda}{C,R}{1\oplus\sigma} \bigr) = \SPHav{\lambda}{C,R}{1\ominus\sigma}$.
\end{lemma}

\begin{proof}
By Proposition \ref{prop:emptyrows} we may assume $R=\emptyset$ without loss of generality.

 Suppose $\lambda$ is a board of minimum width such that for some $\pi\in\SPHav{\lambda}{C,\emptyset}{1\oplus\lambda}$, the image $\pi^* := \zeta_{\lambda}^{C,\emptyset}(\pi)$ contains a copy of $1\ominus\sigma$.  Let $j$ be the index where $\pi_j = 1$.  To the right of column $j$, we see that $\pi_{j+1} \dotsm \pi_n$ must avoid $\sigma$ or else that copy of $\sigma$ would combine with $\pi_j = 1$ to form a copy of $1\oplus\sigma$ in $\pi$.  Since $\pi_{j+1} \dotsm \pi_n \oi \pi^*_{j+1} \dotsm \pi^*_n$, we see that $\pi^*_{j+1} \dotsm \pi^*_n$ also avoids $\sigma$ (and consequently avoids $1\ominus\sigma$).  Thus the copy of $1\ominus\sigma$ appearing in $\pi^*$ may only use indices in $[j]$.  Suppose $\pi^{*}_i \pi^{*}_{j-k+1} \pi^{*}_{j-k+2} \dotsm \pi^{*}_{j}$ forms a copy of $1\ominus\sigma$.  Then $\pi^*_i > \pi^*_j$, but since $\pi^{*}_j = \lambda_j$, the cell $(j, \pi^{*}_i)$ does not lie in $\lambda$.  Thus there can be no copy of $1\ominus\sigma$ in the filling $\pi^*$ which involves $\pi^*_j$.  Thus $\pi^*_1 \dotsm \pi^{*}_{j-1}$ must contain the copy of $1\ominus\sigma$.

By the inductive definition of $\zeta_{\lambda}$, however, we see that $h=\pi_1 \dotsm \pi_{j-1}$ is a filling of $\mu=(\lambda_1, \dotsc, \lambda_{j-1})$ which also avoids $1\oplus\sigma$.  Under the appropriate $\zeta_\mu^{C', R'}$, the filling $h=\pi_1 \dotsm \pi_{j-1}$ maps to $h^{*} = \pi^*_1 \dotsm \pi^{*}_{j-1}$, and so $h$ avoids $1\oplus\sigma$ while its image $h^{*}$ has a copy of $1\ominus\sigma$.  Since $\width{\mu} < \width{\lambda}$, this contradicts minimality of $\width{\lambda}$.  Thus the lemma is proven.
\end{proof}

The proof of Theorem \ref{thm:plustominus} follows immediately from the combination of Lemmas \ref{lem:zetaCRbijection} and \ref{lem:zetaCRavoid}.  The map $\zeta_{\lambda}^{C,R}$ restricts to a bijection between $\SPHav{\lambda}{C,R}{1\oplus\sigma}$ and $\SPHav{\lambda}{C,R}{1\ominus\sigma}$, and therefore $1\oplus\sigma \spswe 1\ominus\sigma$.  This completes the proof of Theorem \ref{thm:plustominus}.

\subsubsection{Implication for Wilf-classification}

Combining Theorem \ref{thm:spswe}, Theorem \ref{thm:plustominus} for $\sigma=12$, and Corollary \ref{cor:BWX} gives us the following Wilf-equivalences for vincular patterns.   This proves Conjecture 17(c) in \cite{Baxter2012}.

\begin{corollary}
 We have the following Wilf-equivalences:
 \begin{equation*}
   12\d3\d4 \we 12\d4\d3 \we 21\d3\d4 \we 21\d4\d3.
 \end{equation*}
\end{corollary}

Computing $\sav{8}{\sigma}$ for all $\sigma$ of length 4 verifies that these patterns and their trivial symmetries make up the entirety of the equivalence class.

Theorem \ref{thm:spswe} in conjunction with Theorems \ref{thm:plustominus} for $\sigma=21$, and the equivalence $1\d23\d4 \we 1\d32\d4$ (first proven in \cite{Elizalde2006}) implies the following corollary:
   \begin{corollary}\label{cor:classH}
     We have the following Wilf-equivalences:
     \begin{equation*}
      3\d12\d4 \we 1\d23\d4 \we 1\d32\d4 \we 3\d21\d4.
     \end{equation*}
   \end{corollary}
  
There appears to be an additional member of this equivalence class.  Computation of $\sav{n}{\sigma}$ for $n\leq 11$ suggests the following conjecture.
   \begin{conjecture}\label{con:classH}
    We have the equivalence $23\d1\d4 \we 1\d23\d4$.
   \end{conjecture}
Note that computation reveals $23\d1 \not\spswe 1\d23$, so Theorem \ref{thm:spswe} cannot apply.

\section{Conclusions and Future Work}\label{sec:Conclusion}

The results above complete the shape-Wilf-classification of vincular patterns of length 3, and contributes towards the Wilf-classification of vincular patterns of length 4 and above.  There appear to be more Wilf-equivalences, such as those discussed above:

\begin{conjecture*}
    The following equivalences are true:
 \begin{itemize}
 \item[] \textbf{Conjecture \ref{con:Conj1}.} $12\d3\d4\d5 \we 12\d4\d3\d5$
 \item[] \textbf{Conjecture \ref{con:classH}.} $23\d1\d4 \we 1\d23\d4$.
 \end{itemize}
\end{conjecture*}

It is not clear whether filling-shape-Wilf-equivalence is a strictly stronger equivalence than shape-Wilf-equivalence.  No examples have been found for pairs of vincular patterns which are shape-Wilf-equivalent but not filling-shape-Wilf-equivalent, so we ask the question: \textit{Is there a pair of patterns $\alpha$ and $\beta$ such that $\alpha\swe\beta$ but $\alpha\not\spswe\beta$?}

Since the results in this paper depend primarily on filling-shape-Wilf-equivalence, it is worth searching for more examples of such equivalences.  Consecutive patterns seem a natural starting point, and Elizalde and Noy find Wilf-equivalences for consecutive patterns of length 4 in \cite{Elizalde2003}.  Four of those equivalences appear to extend to filling-shape-Wilf-equivalence based on an computer search over Young boards $\lambda$ with $\width{\lambda}\leq 8$.
\begin{conjecture}\label{con:cons4}
We have the following filling-shape-Wilf-equivalences for consecutive patterns:
 \begin{enumerate}[(a)]
  \item $1342 \spswe 1432$ 
  \item $2341 \spswe 2431$
  \item $3124 \spswe 3214$
  \item $4123 \spswe 4213$
 \end{enumerate}
\end{conjecture} 
Note that these are all symmetric versions of the same Wilf-equivalence, but since the trivial symmetries do not extend to filling-shape-Wilf-equivalence they must be listed separately.

Last, note that many of the ideas above generalize to the broader context of mesh patterns, which were introduced in \cite{Branden2011} as a generalization of bivincular patterns and vincular patterns.  A mesh pattern $(\sigma, R)$ is a pair $\sigma\in\S{k}$ and $R\subseteq \{0,1,\dotsc,k\} \times \{0,1,\dots,k\}$.  If $(\alpha, R)$ and $(\beta, S)$ are mesh patterns such that $\alpha\in\S{k}$, $\beta\in\S{\ell}$, $(k,k)\notin R$, and $(0,0)\notin S$, then we may define $(\alpha, R) \oplus (\beta,S) := (\alpha\oplus\beta, T)$ where 
\begin{equation*}
 \begin{split}
 T := R &\cup \{(i+k, j+k): (i,j)\in S\} \\
             &\cup \{(i,j): (i, k)\in R, j\in\{k, k+1, \dotsc, k+\ell\} \} \\
             &\cup \{(i+k,j): (i,0)\in S, j\in\{0,1,\dotsc, k\} \} \\
             &\cup \{(i,j): (k,j)\in R, i\in\{k, k+1, \dotsc, k+\ell\} \} \\
             &\cup \{(i,j+k): (0,j)\in S, i\in\{0,1,\dotsc, k\} \}
 \end{split}
\end{equation*}

In words, $T$ is constructed by translating the shaded cells in $S$ to their corresponding positions in $\alpha\oplus\beta$, and the shaded cells on the edges of $(\alpha, R)$ and $(\beta, S)$ expand to fill out their rows and columns.  The shaded cells in $R$ which lie on the right edge of $(\alpha, R)$ expand to shade the remaining portion of the row to the right.  Similarly, shaded cells on left edge of $(\beta, S)$ expand to shade the cells to their left, shaded cells on the top edge of $(\alpha, R)$ expand to shade the cells above them, and shaded cells on the bottom edge of $(\beta, S)$ expand to shade the cells below them.

We illustrate this in Figure \ref{fig:meshsum} with the direct sum of $(312, \{(1,1), (2,3), (3,0), (3,2)\})$ and $(1234, \{(0,2), (2,0), (3,3), (3,4)\})$
\begin{figure}
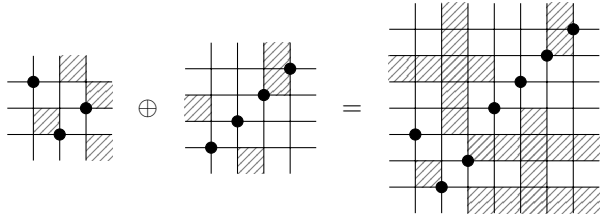

 \centering
\pattern{}{3}{ 1/3, 2/1, 3/2}{1/1, 2/3, 3/0, 3/2} 
$\oplus$
\pattern{}{4}{ 1/1, 2/2, 3/3, 4/4}{0/2, 2/0, 3/3, 3/4}
$=$
\pattern{}{7}{1/3, 2/1, 3/2,  4/4, 5/5, 6/6, 7/7}{  1/1, 2/3, 3/0, 3/2 ,    3/5, 5/3, 6/6, 6/7,     2/4, 2/5, 2/6, 2/7,    5/0, 5/1, 5/2,  3/2, 4/2, 5/2, 6/2, 7/2,    0/5, 1/5, 2/5 ,   4/0, 5/0, 6/0, 7/0 }
 \caption{The direct sum of two mesh patterns}
 \label{fig:meshsum}
\end{figure}

 It is left to the reader to verify that the proof of Theorem \ref{thm:spswe} extends to the case that $\alpha$, $\beta$, and $\sigma$ are mesh patterns such that $\alpha\oplus\sigma$ and $\beta\oplus\sigma$ are well-defined.  This suggests that a search for filling-shape-Wilf-equivalence over mesh patterns could yield infinite families of Wilf-equivalences among the mesh patterns.  Such a search has yet to be carried out.

\bigskip

{\bf Acknowledgements:} The author would like to thank Einar Steingr\'{i}msson, Alexander Burstein, and Lara Pudwell for helpful comments during the preparation of this article.


\begin{thebibliography}{10}

\bibitem{Babson2000}
Eric Babson and Einar Steingr{\'{\i}}msson.
\newblock Generalized permutation patterns and a classification of the
  {M}ahonian statistics.
\newblock {\em S\'em. Lothar. Combin.}, 44:Art. B44b, 18 pp. (electronic),
  2000.

\bibitem{Backelin2007}
J{\"o}rgen Backelin, Julian West, and Guoce Xin.
\newblock Wilf-equivalence for singleton classes.
\newblock {\em Adv. in Appl. Math.}, 38(2):133--148, 2007.

\bibitem{Baxter2012}
Andrew~M. Baxter and Lara~K. Pudwell.
\newblock Enumeration schemes for vincular patterns.
\newblock {\em Discrete Math.}, 312(10):1699--1712, 2012.

\bibitem{Bousquet2010}
Mireille Bousquet-M{\'e}lou, Anders Claesson, Mark Dukes, and Sergey Kitaev.
\newblock {$(2+2)$}-free posets, ascent sequences and pattern avoiding
  permutations.
\newblock {\em J. Combin. Theory Ser. A}, 117(7):884--909, 2010.

\bibitem{Branden2011}
Petter Br{\"a}nd{\'e}n and Anders Claesson.
\newblock Mesh patterns and the expansion of permutation statistics as sums of
  permutation patterns.
\newblock {\em Electron. J. Combin.}, 18(2):Paper 5, 14, 2011.

\bibitem{Claesson2011}
Anders Claesson, V{\'{\i}}t Jel{\'{\i}}nek, Eva Jel{\'{\i}}nkov{\'a}, and
  Sergey Kitaev.
\newblock Pattern avoidance in partial permutations.
\newblock {\em Electron. J. Combin.}, 18(1):Paper 25, 41, 2011.

\bibitem{Elizalde2006}
Sergi Elizalde.
\newblock Asymptotic enumeration of permutations avoiding generalized patterns.
\newblock {\em Adv. in Appl. Math.}, 36(2):138--155, 2006.

\bibitem{Elizalde2003}
Sergi Elizalde and Marc Noy.
\newblock Consecutive patterns in permutations.
\newblock {\em Adv. in Appl. Math.}, 30(1-2):110--125, 2003.
\newblock Formal power series and algebraic combinatorics (Scottsdale, AZ,
  2001).

\bibitem{Kitaev2005}
Sergey Kitaev.
\newblock Partially ordered generalized patterns.
\newblock {\em Discrete Math.}, 298(1-3):212--229, 2005.

\bibitem{Rubey2011}
Martin Rubey.
\newblock Increasing and decreasing sequences in fillings of moon polyominoes.
\newblock {\em Adv. in Appl. Math.}, 47(1):57--87, 2011.

\bibitem{Stankova2002}
Zvezdelina Stankova and Julian West.
\newblock A new class of {W}ilf-equivalent permutations.
\newblock {\em J. Algebraic Combin.}, 15(3):271--290, 2002.

\end{thebibliography}

\end{document}